\documentclass[11pt,reqno]{amsart}

\usepackage{amsfonts}
\usepackage{amsmath}
\usepackage{amsthm}
\usepackage{amssymb}
\usepackage[all]{xy}
\usepackage{enumerate}
\usepackage{fullpage}
\usepackage[numbers]{natbib}
\usepackage[colorlinks=true, citecolor=green, linkcolor=red, bookmarks=true]{hyperref}

\theoremstyle{plain}
\newtheorem{theorem}{Theorem} 
\newtheorem{lemma}[theorem]{Lemma}

\theoremstyle{definition}
\newtheorem{definition}[theorem]{Definition}

\theoremstyle{remark}
\newtheorem{remark}[theorem]{Remark}

\newcommand{\dist}{\textrm{dist}}

\title{Decomposable approximations and approximately finite dimensional $C^*$-algebras}
\author{Jorge Castillejos} 
\address{School of Mathematics and Statistics, University of Glasgow, Glasgow G12 8QW, UK.}
\email{j.castillejos-lopez.1@research.gla.ac.uk}
\thanks{Author supported by the Mexican National Council of Science and Technology (CONACYT)}
\date{\today}

\begin{document}

\maketitle

\begin{abstract}
Nuclear $C^*$-algebras having a system of completely positive approximations formed with convex combinations of a uniformly bounded number of order zero summands are shown to be approximately finite dimensional.

\end{abstract}

\section{Introduction}

The nuclear dimension is a non-commutative theory of covering dimension for nuclear $C^*$-algebras introduced by Winter and Zacharias in \cite{nuclear-dimension}, which extends the earlier notion of decomposition rank from \cite{Kir-Win}.  These concepts have played a key role in the recent revolutionary progress in the structure theory of simple nuclear $C^*$-algebras, such as Winter's $\mathcal{Z}$-stability theorems \cite{monstruosity-I,monstruosity-II} which show that simple separable unital nuclear $C^*$-algebras of finite non-commutative covering dimension have the striking algebraic property of tensorially absorbing the Jiang-Su algebra $\mathcal{Z}$.  This forms part of the Toms-Winter regularity conjecture which seeks to characterize those simple nuclear $C^*$-algebras accessible to classification (\textit{cf}.\@ \cite{regularity-prop}) through topological dimension, $\mathcal{Z}$-absorption and the structure of the Cuntz semigroup, and there have been a number of high profile developments relating these properties (\cite{pre-Matui-Sato,matui-sato,aaron-wilhelm}) including a recent converse to \cite{monstruosity-II} in the unique trace case \cite{sato-winter-white}.

As shown by Kirchberg \cite{kirchberg-CPAP} and Choi-Effros \cite{choi-effros-CPAP}, nuclearity can be defined using the completely positive approximation property (CPAP). For commutative $C^*$-algebras the
CPAP is established from partitions of unity subordinate to suitable
open covers of the spectrum $X$ of the algebra. When $X$ is finite
dimensional, these covers can be taken to be finitely coloured, and
this can be seen in additional properties of the resulting
approximation: the maps approximating $C_0(X)$ by finite dimensional
algebras are \emph{finitely decomposable}.  Precisely, when $F$ is
finite dimensional, a completely positive map $\varphi:F\rightarrow A$ is $n$-\emph{decomposable} if there exists a natural number $n$ such that we can express $F = \bigoplus\limits_{k=0}^{n} F_k$ and the restrictions $\varphi |_{F_k}$ are order zero, \textit{i.e.}$\!$ preserve orthogonality (\textit{cf}.\@ \cite{order-zero}). Decomposition rank is defined by asking for completely positive and contractive approximations of the identity map $\mathrm{id}_A$ of the form
$$
A\stackrel{\psi}{\longrightarrow}F\stackrel{\varphi}{\longrightarrow}A
$$
where $\varphi$ is $n$-decomposable for some fixed natural number $n$. The minimum $n$ with this property is the value of the decomposition rank of $A$. Nuclear dimension is defined in a similar way but without requiring $\varphi$ to be contractive. The zero dimensional $C^*$-algebras for these theories are precisely the approximately finite dimensional algebras \cite[Theorem 3.4]{cov-dim} in contrast with other notions of dimension such as the real rank.

A stronger version of the completely positive approximation property was established in 2012 in \cite[Theorem 1.4]{Hir-Kir-Whi}. This shows that the maps
$\varphi$ can always be taken to be decomposable, though the size of the
decomposition may vary with the tolerances in the approximation. Moreover, this theorem shows that these approximations can be taken as a convex combination of contractive order zero maps and this is a crucial ingredient in obtaining a near inclusion type perturbation result for separable nuclear $C^*$-algebras \cite[Section 2]{Hir-Kir-Whi}. Thus, as suggested by Winter in the NSF/CBMS conference in Louisiana 2012 it is natural to investigate the situation when the completely positive and contractive approximations are decomposable as convex combination with a uniformly bounded number of summands. In this note we show that such approximations force the underlying $C^*$-algebra to be approximately finite dimensional (Theorem \ref{main theo}).

\section{Preliminaries}

In this section we are going to recall all the definitions and properties that we will use in the next section. We will denote the set of positive elements of the $C^*$-algebra $A$ as $A_+$ and $A_+^1$ will denote the set of contractive positive elements in the algebra. All the direct sums will be regarded as internal direct sums and, as usual, for $a \in A$ and $X \subset A, \; \dist \left(a,X\right)$ will denote $\inf\limits_{x \in X} \left\| a-x \right\|$. 

Approximately finite dimensional $C^*$-algebras were defined originally by Bratteli (\cite[Definition 1.1]{bratelli-AF}). A $C^*$-algebra $A$ is approximately finite dimensional (AF) if it contains an increasing sequence of finite dimensional $C^*$-algebras $\left\{ A_n \right\}_{n \in \mathbb{N}}$ such that $\bigcup\limits_{n \in \mathbb{N}} A_n$ is dense in $A$. 
It is a consequence of the definition that AF-algebras are separable and Bratteli proved the following theorem, known as the local characterisation of AF-algebras (\cite[Theorem 2.2]{bratelli-AF}).

\begin{theorem}[Bratteli]
	A separable $C^*$-algebra $A$ is AF if and only if for every finite subset $\mathfrak{F} \subset A$ and $\varepsilon > 0$ there exists a finite dimensional $C^*$-algebra $B \subset A$ such that 
	$$
	\dist \left(a, B\right) < \varepsilon
	$$
	for all $a \in \mathfrak{F}$.
\end{theorem}

Winter proved that a separable $C^*$-algebra has nuclear dimension $0$ if and only if it is AF using the local characterisation (\cite[Remark 2.2.(iii)]{nuclear-dimension}). There are two possible definitions of non separable AF-algebras, either as algebras containing a directed family of finite dimensional $C^*$-subalgebras with dense union (equivalently as the direct limit of finite dimensional $C^*$-algebras over general directed sets) or via the local characterisation. These are not the same (\cite[Theorem 1.5]{farah-AF}), so in this paper we choose to work with the local characterisation as the definition of AF, so that AF-algebras are precisely those with nuclear dimension $0$.  

\begin{definition}
	A $C^*$-algebra is AF if for every finite subset $\mathfrak{F} \subset A$ and $\varepsilon > 0$ there exists a finite dimensional $C^*$-algebra $B \subset A$ such that 
	$$
	\dist \left(a, B\right) < \varepsilon
	$$
	for all $a \in \mathfrak{F}$. 
\end{definition}


The multiplier algebra, $\mathcal{M}(A)$, is the $C^*$-analogue of the Stone-\v{C}ech compactification. For our purposes we use the original construction, due to Busby \cite{busby}, using double centralizers.

\begin{definition}
Let $A$ be a $C^*$-algebra. A \emph{double centralizer} is a pair $(L,R)$ of maps $L,R:A \longrightarrow A$ such that $aL(b)=R(a)b$ for all $a,b \in A$. $\mathcal{M}(A)$ will denote the set of double centralizers of $A$.
\end{definition}

One can then define operations on $\mathcal{M}(A)$ in order to equip it with the structure of a unital $C^*$-algebra \cite[Definition 2.10, Theorem 2.11]{busby}. Moreover, we have an embedding $\mathfrak{M}:A \longrightarrow \mathcal{M}(A)$ given by
\begin{equation}
\mathfrak{M}_a = (L_a, R_a)
\end{equation} 
where $L_a$ and $R_a$ are defined as left and right multiplication by $a \in A$, respectively.

The following lemma will be useful in the proof of Lemma \ref{lemma 2 omega}. This is where the hypothesis of having convex combinations is used.

\begin{lemma}\label{multiplier lemma}
Let $A$ be a $C^*$-algebra and $a_1,a_2 \in A_+^1$. Let $B$ be a $C^*$-subalgebra of $A$ and let $\lambda_1$ and $\lambda_2$ be strictly positive real numbers satisfying $\lambda_1 + \lambda_2 = 1$. If $a_1 b \in B$ and $\left(\lambda_1 a_1 + \lambda_2 a_2 \right) b = b$ for all $b \in B$, then $a_1 b = a_2 b = b$ for all $b \in B$.  
\end{lemma}

\begin{proof}
Using the hypothesis we have that $\mathfrak{M}_{a_1} = \left(L_{a_1}|_B, R_{a_1}|_B\right) \in \mathcal{M}(B)$. Similarly if $a = \lambda_1 a_1 + \lambda_2 a_2$ then $\mathfrak{M}_a \in \mathcal{M}(B)$. In fact, $\mathfrak{M}_a = 1_{\mathcal{M}(B)}$. We have, for all $b \in B$,
\begin{equation}
\lambda_2 a_2 b =  b - \lambda_1 a_1 b.
\end{equation}
By the hypothesis, the right side of the previous equation is in $B$, therefore $a_2 b \in B$ for all $b \in B$ and this yields $\mathfrak{M}_{a_2} \in \mathcal{M}(B)$. It is also straightforward to see that
\begin{equation}
1_{\mathcal{M}(B)} = \mathfrak{M}_a = \lambda_1 \mathfrak{M}_{a_1} + \lambda_2 \mathfrak{M}_{a_2}.
\end{equation}
By \cite[Theorem II 3.2.17]{blackadar}, $1_{\mathcal{M}(B)}$ is an extreme point of the unit ball of $\mathcal{M}(B)$. Since $\mathfrak{M}_{a_1}$ and $\mathfrak{M}_{a_2}$ lie in the unit ball we have
\begin{equation*}
1_{\mathcal{M}(B)} = \mathfrak{M}_{a_1} = \mathfrak{M}_{a_2}. \tag*{\qedhere}
\end{equation*}
\end{proof}

The next technical lemma will be used in the proof of the main theorem and it will allow us to work with one order zero map instead of a convex combination.

\begin{lemma}\label{lemma 3 omega}
Let $A$ be a $C^*$-algebra, $\varepsilon > 0$ and let $\left(\lambda_k\right)_{k \in \mathbb{N}}$ be a sequence contained in $[0,1]$ such that $\sum\limits_{k=1}^\infty \lambda_k =1$. If $p \in A$ is a projection and $a_k \in A_+^1$, $k \in \mathbb{N}$, satisfy
\begin{equation} \label{p - convex comb}
\left\| p - \sum_{k} \lambda_k a_k \right\| \leq \varepsilon.
\end{equation}
Then 
\begin{equation}
\left\| p - a_k \right\| \leq \sqrt{\lambda_k^{-1} \varepsilon} \left(\sqrt{\lambda_k^{-1} \varepsilon} + 1 \right) 
\end{equation}
for any $\lambda_k \neq 0$.
\end{lemma}

\begin{proof}
We may suppose $A \subset B(H)$ for some Hilbert space $H$. For fixed $k$ consider
\begin{equation}
b= \frac{1}{1 - \lambda_k} \sum_{i \neq k} \lambda_i a_i \in A_+^1.
\end{equation}
With this construction we can treat the sum as the convex combination of only two summands, precisely
\begin{equation}
\sum\limits_{i} \lambda_i a_i = \lambda_k a_k + \left(1 - \lambda_k \right)b. 
\end{equation}
By (\ref{p - convex comb}), we get $p - \left( \lambda_ka_k + \left(1 - \lambda_k \right)b \right) \leq \varepsilon 1_{B(H)}$. Thus
$\lambda_k \left(p - pa_kp\right) + (1-\lambda_k)\left(p - pbp \right) \leq \varepsilon p$. Since $p-pa_kp$ and $p - pbp$ are positive, the previous inequality leads to
\begin{equation}
0 \leq p - pa_kp \leq \lambda_k^{-1} \varepsilon p - \left(\lambda_k^{-1} -1\right)\left(p-pbp \right) \leq \lambda_k^{-1} \varepsilon p.
\end{equation}
Thus
\begin{equation}
\|p-p a_k p\| \leq \lambda_k^{-1}\varepsilon 
\end{equation}
and similarly we obtain
\begin{equation}
\|(1_{B(H)}-p)a_k (1_{B(H)}-p)\| \leq \lambda_k^{-1}\varepsilon. \label{norm 1-p a 1-p}
\end{equation}
 
We can write any $h \in H$ as $h_1 + h_2$ where $h_1 = p(h)$ and $h_2 = (1_{B(H)} - p)(h)$. Since $a_k$ is positive, we have
\begin{align}
0  &\leq \langle a_k h , h \rangle\\ 
  &=  \langle pa_k p(h_1),h_1 \rangle + 2 Re \langle pa_k(1_{B(H)} -p)\left(h_2\right),h_1 \rangle + \langle (1_{B(H)}-p)a_k(1_{B(H)}-p)(h_2),h_2\rangle. \label{fea}
\end{align}

Let us suppose that $\| pa_k (1_{B(H)}-p) \| > \sqrt{\lambda_k^{-1}\varepsilon}$. Then
there exists $h_2 \in (1_{B(H)} - p)\left(H\right)$ with $\| h_2 \|= 1$ such that  $\| pa_k(1_{B(H)}-p)(h_2) \| > \sqrt{\lambda_k^{-1}\varepsilon}$. Set $h_1 = pa_k(1_{B(H)}-p)(h_2)$ and considering $h=- h_1 + h_2$ in (\ref{fea}) we obtain
\begin{eqnarray}
0 &\leq& \langle pa_k p(-h_1),-h_1 \rangle + 2 Re \langle pa_k(1_{B(H)} -p)(h_2),-h_1 \rangle \\
& & 
+ \; \langle (1_{B(H)}-p)a_k(1_{B(H)}-p)(h_2),h_2\rangle \\
& \overset{(\ref{norm 1-p a 1-p})}{\leq} & \langle p(h_1),h_1 \rangle - 2 \langle h_1,h_1 \rangle + \lambda_k^{-1}\varepsilon \\
& = & - \|h_1\|^2 + \lambda_k^{-1}\varepsilon \\
& < & - \lambda_k^{-1}\varepsilon + \lambda_k^{-1}\varepsilon = 0
\end{eqnarray}
which is clearly a contradiction. Therefore 
\begin{equation}
\| (1_{B(H)}-p)a_k p \|= \| pa_k (1_{B(H)} - p) \| \leq \sqrt{\lambda_k^{-1}\varepsilon}. 
\end{equation}
Finally we obtain
\begin{align}
\| p - a_k \| \leq & \max\left\{\| p - pa_kp \|, \| (1_{B(H)}-p)a_k(1_{B(H)}-p) \|\right\} \\
& + \; \max\left\{ \| pa_k(1_{B(H)} -p)\| ,  \|(1_{B(H)} -p)a_k p \| \right\}\\
 \leq & \, \lambda_k^{-1} \varepsilon + \sqrt{\lambda_k^{-1} \varepsilon} = \sqrt{\lambda_k^{-1} \varepsilon} \left( \sqrt{\lambda_k^{-1} \varepsilon} + 1 \right). \tag*{\qedhere}
\end{align}
\end{proof}

We will refer to a completely positive map as a CP map and, similarly, a completely positive and contractive map as a CPC map. Let us now recall the definition of order zero maps introduced by Winter and Zacharias in \cite{order-zero}.

\begin{definition}
A CP map $\varphi:A \longrightarrow B$ between $C^*$-algebras has \emph{order zero} if it preserves orthogonality; \textit{i.e.}$\!$ if $a,b \in A_{+}$ satisfy $ab=0$ then $\varphi(a)\varphi(b)=0$.
\end{definition}

Based on a result of Wolff (\cite[Theorem 2.3]{wolff}), 
Winter and Zacharias proved in \cite[Theorem 3.3]{order-zero} the following structure theorem for CP maps of order zero.

\begin{theorem}\label{order zero}
Let $\varphi: A \longrightarrow B$ a CP map of order zero between $C^*$-algebras and set $C:= C^*\left(\varphi\left(A\right)\right)$. Then there exist a positive $h \in \mathcal{M}\left(C\right) \cap C'$ with $\|h\|=\|\varphi\|$ and a $^*$-homomorphism $$\rho: A \longrightarrow \mathcal{M}\left(C\right)\cap \{h\}'$$ such that
\begin{equation}
\varphi(a) = h \rho(a) 
\end{equation}
for all $a \in A$. If $A$ is unital, then one may take $h = \varphi(1_A)$.
\end{theorem}


%

The proof of the following lemma is essentially the proof of \cite[Proposition 3.2 (c)]{cov-dim}. 

\begin{lemma}\label{pi - varphi}
	For every $\delta >0$ there exists $\gamma >0$ such that for any CPC order zero map $\varphi: A \longrightarrow B$ between $C^*$-algebras, with $A$ unital, satisfying 
	$$
	\left\| \varphi\left(1_A\right) - \varphi\left(1_A\right)^2 \right\| < \gamma
	$$
	there exists a $^*$-homomorphism $\pi: A \longrightarrow B$ such that
	$$
	\left\| \varphi - \pi \right\| < \delta.
	$$
\end{lemma}

\begin{proof}
	Consider $\gamma < \min\{ \varepsilon/2 , 1/4 \}$. Then by \cite[Proposition 2.17]{cov-dim} there exists a projection $p \in C^*\left(\varphi\left(1_A\right)\right)$ such that $\|p - \varphi\left(1_A\right)\| < \varepsilon$. By Theorem \ref{order zero}, there exists a $^*$-homomorphism $\rho: A \longrightarrow \mathcal{M}\left(C^*\left(\varphi\left(A\right)\right)\right)\cap \{\varphi\left(1_A\right)\}'$ such that $\varphi\left(a\right) = \varphi\left(1_A\right)\rho\left(a\right)$ for all $a \in A$. Set $\pi: A \longrightarrow B$ as $\pi\left(a\right)= \rho \left(a\right) p$. 
	As $p \in C^*\left( \varphi\left(1_A\right) \right) \subset \rho \left(A\right)'$, this defines an order zero map with $\pi\left(1_A\right)=p$ and 
	\begin{equation}
	\left\| \varphi - \pi \right\| \leq \left\| \varphi\left(1_A\right) - p \right\| < \varepsilon.
	\end{equation}
	Finally, by \cite[Proposition 3.2 (b)]{cov-dim} $\pi$ is a $^*$-homomorphism.
\end{proof}

Given a sequence of $C^*$-algebras $\left\{ A_n \right\}_{n \in \mathbb{N}}$, set
\begin{equation}
\ell^\infty \left( \left\{ A_n \right\}_{n \in \mathbb{N}} \right) = \left\{ \left(a_n\right)_{n \in \mathbb{N}} \; \middle| \; a_n \in A_n\,,\; \sup_{n \in \mathbb{N}} \left\| a_n \right\| < \infty \right\}. 
\end{equation}

\begin{definition}
Let $\left\{ A_n \right\}_{n \in \mathbb{N}}$ be a sequence of $C^*$-algebras and $\mathcal{U}$ a filter on $\mathbb{N}$. We define the sequence algebra of $\left\{ A_n \right\}_{n \in \mathbb{N}}$ as
\begin{equation}
\prod_{n \rightarrow\infty} A_n = \ell^\infty \left( \left\{ A_n \right\}_{n \in \mathbb{N}} \right) \Big/  \left\{ \left(a_n\right)_{n \in \mathbb{N}} \in \ell^\infty \left( \left\{ A_n \right\}_{n \in \mathbb{N}} \right) \; \middle| \; \lim\limits_{n \rightarrow \infty} \|a_n\| = 0 \right\}.
\end{equation}
We also define $\prod\limits_{n \rightarrow \mathcal{U}} A_n $ as
\begin{equation}
\prod_{n \rightarrow \mathcal{U}} A_n = \ell^\infty \left( \left\{ A_n \right\}_{n \in \mathbb{N}} \right) \Big/ \left\{ \left(a_n\right)_{n \in \mathbb{N}} \in \ell^\infty \left( \left\{ A_n \right\}_{n \in \mathbb{N}} \right) \; \middle| \; \lim\limits_{n \rightarrow \mathcal{U}} \|a_n\| = 0 \right\}.
\end{equation}
\end{definition}

\noindent We will omit the $n$ when there is no risk of confusion. If $A$ is a $C^*$-algebra and $A_n = A$ for all $n \in \mathbb{N}$, we denote them as $A_\infty$ and $A_\mathcal{U}$.  When $\mathcal{U}$ is an ultrafilter, $\prod\limits_{\mathcal{U}} A_n$ is called an ultraproduct and $A_\mathcal{U}$ an ultrapower.
We can identify $A$ as a subalgebra of $A_\mathcal{U}$ via the canonical embedding as constant sequences.

Consider a $C^*$-algebra $A$, a finite subset $\mathfrak{F}\subset A$ and $\varepsilon > 0$. 
A CPC approximation for $\mathfrak{F}$ within $\varepsilon$ is an ordered triple $\left(F, \psi, \varphi \right)$ where $\psi: A \longrightarrow F$ and $\varphi: F \longrightarrow A$ are CPC maps and $F$ is a finite dimensional $C^*$-algebra satisfying $\|a - \varphi \psi (a) \|< \varepsilon$ for all $a \in \mathfrak{F}$. 

A system of CPC approximations for $A$ will be a net of CPC approximations $\left(F^{(r)}, \psi^{(r)}, \varphi^{(r)} \right)$ 
converging to $\mbox{id}_A$ in the point-norm topology. It will be denoted as $\left\{ \left(F^{(r)} , \psi^{(r)}, \varphi^{(r)} \right) \right\}_{r \in I}$. If $A$ is separable, it is enough to consider a sequence of CPC approximations. 
The proof of the following lemma is contained in the proof of \cite[Lemma 3.7]{Kir-Win}.

\begin{lemma} \label{missing lemma}
	Let $A$ be a separable nuclear $C^*$-algebra. Let $\left\{ \left(F^{(r)}, \psi^{(r)}, \varphi^{(r)}\right) \right\}_{r \in \mathbb{N}}$ be a system of CPC approximations for $A$ with $F^{(r)}$ finite dimensional. Suppose $0<\varepsilon\leq 1$ is given and let $\mathfrak{F} \subset A_+$ be a finite subset. 
	Then there exists $r \in \mathbb{N}$ and a projection $p \in F^{(r)}$ such that
	\begin{align}
		\left\| \varphi^{(r)} \psi^{(r)}(a)  - a \right\| < \varepsilon
	\end{align}	
	 and
	 \begin{align}
		\left\| \varphi^{(r)} \left(p \psi^{(r)}(a) p \right) - a \right\| < \varepsilon
	\end{align}
	for all $a \in \mathfrak{F}$. Moreover, if $F^{(r)}= \bigoplus\limits_{k=1}^n F^{(r)}_k$ and $p_k = p 1_{F_k}$ then
	\begin{equation}
	\left\| \varphi^{(r)}\left(p_k\right) - \varphi^{(r)}\left(p_k\right)\varphi^{(r)}\left( 1_{F^{(r)}} \right) \right\| < \varepsilon 
	\end{equation}
	for $k=1, \cdots , n$.
\end{lemma}

\section{The Main Result}

We will now proceed to prove the main theorem. We will split the proof in two steps. Firstly, we show that the order zero maps appearing in the convex combinations can be replaced by $^*$-homomorphisms, and secondly, by approximating twice in a suitable way, we use these to obtain the finite dimensional approximations. 

The following lemma will be given in greater generality than is needed for the proof of the main theorem of this section in order to also use it in the next section. Since we do not need any special feature of ultrafilters, it will be enough to work with free filters. This will allow us later, to apply the following lemma in the situation of nuclear dimension at most $\omega$ where we will work with the sequence algebra.

\begin{lemma}\label{lemma 2 omega}
Let $A$ be a separable $C^*$-algebra, $\mathcal{U}$ a free filter on $\mathbb{N}$ containing the cofinite filter. Let $\left(\lambda_k\right)_{k \in \mathbb{N}}$ be a sequence contained in $[0,1]$ such that $\sum\limits_{k=1}^{\infty}\lambda_k = 1$ and let $\left\{ a_n \right\}_{n \in \mathbb{N}}$ be a dense countable subset of $A$. Suppose $A$ has a system of CPC approximations $\left\{ \left(F^{(r)}, \psi^{(r)}, \varphi^{(r)} \right) \right\}_{r \in \mathbb{R}}$ satisfying the following conditions:
\begin{enumerate}[(a)]

\item For every $r \in \mathbb{N}$ there exist $n^{(r)} \in \mathbb{N}$ with $n^{(r)} \leq m$, 
a decomposition $F^{(r)} = \bigoplus\limits_{k=1}^{n^{(r)}} F_k^{(r)}$, as internal direct sum, and
a family $\left\{\varphi_k^{(r)}:F^{(r)} \longrightarrow A : k \in \mathbb{N} \right\}$ of contractive order zero maps such that $\varphi_k^{(r)} = 0$ if $k > n^{(r)}$ satisfying
\begin{equation}
\varphi^{(r)} = \sum\limits_{k=1}^{n^{(r)}} \lambda_k \varphi_k^{(r)}.
\end{equation}
Moreover, $\bigoplus\limits_{i \neq k} F_i^{(r)} \subset \ker \varphi_k$.

\item For every $r \in \mathbb{N}$ there exist projections $p_k^{(r)} \in F_k^{(r)}$ satisfying 

\begin{enumerate}[(I)]

\item $\left\| \varphi^{(r)} \psi^{(r)} (a_n) - a_n \right\| < r^{-1} $ for $n \leq r$.

\item $\left\| \varphi^{(r)}\left(p^{(r)} \psi^{(r)} \left(a_n\right) p^{(r)}\right) - a_n \right\| < r^{-1} $ for $n \leq r$ with $p^{(r)}= \sum\limits_{k=1}^{n^{(r)}} p_k^{(r)}$. \label{(I)}

\item $\left\| \varphi^{(r)} \left(p_k^{(r)}\right) - \varphi^{(r)}\left(p_k^{(r)}\right)\varphi^{(r)}\left(1_{F^{(r)}}\right) \right\| < r^{-1}$ where $1_{F^{(r)}}$ denotes the unit of $F^{(r)}$. \label{(II)}
	
\end{enumerate} 


\end{enumerate}
 
Then for every finite subset $\mathfrak{F}\subset A$ and every $\varepsilon>0$ there exist $N \in \mathbb{N}$ and a CPC approximation $\left( \bigoplus\limits_{k=1}^N \widetilde{F}_k, \psi, \pi \right)$ for $\mathfrak{F}$ within $\varepsilon$
such that $\pi = \sum\limits_{k=1}^{N}\lambda_k \pi_k$ with each $\pi_k:\bigoplus\limits_{k=1}^N \widetilde{F}_k \longrightarrow A$ a $^*$-homomorphism satisfying $\bigoplus\limits_{i \neq k} \widetilde{F}_i \subset \ker \pi_k$

\end{lemma}

\begin{proof}
	Let $\mathfrak{F} \subset A$ and $\varepsilon >0$. Without losing generality we can assume the elements of $\mathfrak{F}$ are in the dense subset $\left\{a_n \right\}$ and are positive contractions. Consider $\gamma$ given by Lemma \ref{pi - varphi} using $\delta = \varepsilon/3$. Since $\sum\limits_{k=1}^{\infty}\lambda_k = 1$ there exists $N \in \mathbb{N}$ such that
	\begin{equation}
	\sum\limits_{k>N}^{\infty} \lambda_k < \frac{\varepsilon}{3}. \label{reduction series}
	\end{equation}
	We will show below that $\varphi_k^{(\mathcal{U})} \left(p_k\right)$ is a projection for every $k$ with $\lambda_k \neq 0$, where $p_k \in \prod\limits_{\mathcal{U}} F^{(r)}$ is represented by $\left(p_k^{(r)}\right)_{r \in \mathbb{N}}$. Once this is done, for any $k \in \mathbb{N}$, there exists $U_k \in \mathcal{U}$ such that
	\begin{equation}\label{gamma}
	\left\| \varphi_k^{(r)} \left( p_k^{(r)} \right)-\varphi_k^{(r)} \left( p_k^{(r)} \right)^2 \right\| < \gamma
	\end{equation}
	for all $r \in U_k$.
	Similarly, since $\lim\limits_{r \rightarrow \mathcal{U}} \varphi^{(r)} \psi^{(r)} \left(p^{(r)} a_n p^{(r)}\right) = a_n$ for all $n \in \mathbb{N}$, there exists $V \in \mathcal{U}$ such that
	\begin{equation}
	\left\| a - \varphi^{(r)}\left( p^{(r)} \psi^{(r)} \left( a \right) p^{(r)} \right) \right\| < \varepsilon/3 \label{x - phi psi (p x p)}
	\end{equation}
	for all $r \in V$ and for all $a \in \mathfrak{F}$.
	
	
	Fix $r \in U_1 \cap \cdots U_N \cap V$ and set $\widetilde{F}_k = p_k^{(r)}F^{(r)} p_k^{(r)}$. Hence, by the choice of the constant $\gamma$ and (\ref{gamma}), there exists a $^*$-homomorphism $\pi_k: \widetilde{F}_k \longrightarrow A $ such that
	\begin{equation}
	\left\| \varphi_k^{(r)}|_{\widetilde{F}_k} - \pi_k \right\| < \frac{\varepsilon}{3} \label{reduction order zero with hom}
	\end{equation}
	for $k \leq N$. Extend $\pi_k$ to $\widetilde{F}:= \bigoplus\limits_{i=1}^{n^{(r)}} \widetilde{F_k} = p^{(r)} F^{(r)}p^{(r)} $ linearly by defining $\pi_k(x_1 \oplus \cdots \oplus x_{k-1} \oplus 0 \oplus x_{k+1} \oplus \cdots \oplus x_n)=0$ for $x_i \in \widetilde{F}_i$ with $i \neq k$.
	
	Define $\psi:A \longrightarrow \widetilde{F}$ as $\psi(a) = p^{(r)} \psi^{(r)}(a) p^{(r)}$ and set $\pi: \widetilde{F} \longrightarrow A$ as $\pi = \sum\limits_{k=1}^{N} \lambda_k \pi_k$, then $\left(\widetilde{F}, \psi, \pi \right)$ is a completely positive and contractive approximation with the required properties since, using (\ref{x - phi psi (p x p)}), (\ref{reduction order zero with hom}) and (\ref{reduction series}), we obtain
	\begin{align}
	\left\| a - \pi\psi(x) \right\| \leq& \left\| a - \varphi^{(r)}\left(p^{(r)} \psi^{(r)}(a) p^{(r)}\right) \right\| + \left\| \sum\limits_{k=1}^{N} \lambda_k \left(\varphi_k^{(r)} - \pi_k\right)\left(p^{(r)} \psi^{(r)}(a) p^{(r)} \right) \right\| \\ 
	& + \left\| \sum\limits_{k>N} \lambda_k \varphi_k^{(r)} \left( p^{(r)} \psi^{(r)}(a)p^{(r)} \right) \right\| \\
	<& \: \frac{\varepsilon}{3} + \sum_{k=1}^{N}\lambda_k \left( \frac{\varepsilon}{3} \right) + \frac{\varepsilon}{3}  < \varepsilon
	\end{align}
	for all $a \in \mathfrak{F}$.
	
	To finish the proof, we will show $\varphi_k^{(\mathcal{U})}\left( p_k \right)$ is a projection for every $k \in \mathbb{N}$ with $\lambda_k=0$. Due to the hypotheses, we have $\varphi^{(\mathcal{U})} = \sum\limits_{k=1}^{\infty} \lambda_k \varphi_k^{(\mathcal{U})} $
	and $\varphi^{(\mathcal{U})} \psi^{(\mathcal{U})} (a) = a$ for all $a \in A$. Let us remember $p_k \in \prod\limits_\mathcal{U} F^{(r)}$ is represented by $\left( p_k^{(r)} \right)_r $ and consider $p \in \prod\limits_\mathcal{U} F^{(r)}$ represented by $\left( p^{(r)} \right)_{r}$ with $p^{(r)} = \sum\limits_{k=1}^{n^{(r)}} p_k^{(r)}$, then by (\ref{(I)}) we have
	\begin{eqnarray}
	\varphi^{(\mathcal{U})}(p\psi^{(\mathcal{U})}(a)p) = a
	\end{eqnarray}
	and by (\ref{(II)})
	\begin{equation}
	\varphi^{(\mathcal{U})}(p_k)=\varphi^{(\mathcal{U})}(p_k)\varphi^{(\mathcal{U})}(1_{\prod\limits_\mathcal{U} F^{(r)}})\label{useless}
	\end{equation}
	for all $a \in A$ where $1_{\prod\limits_\mathcal{U} F^{(r)}}$ denotes the unit of $\prod\limits_\mathcal{U} F^{(r)}$. Taking adjoints in (\ref{useless}) we get 
	\begin{equation}
	\varphi_k^{(\mathcal{U})} (p_k) =\varphi_k^{(\mathcal{U})}(p_k)\varphi^{(\mathcal{U})}(1_{\prod\limits_\mathcal{U} F^{(r)}})= \varphi^{(\mathcal{U})} (1_{\prod\limits_\mathcal{U} F^{(r)}}) \varphi_k^{(\mathcal{U})}(p_k).
	\end{equation}
	
	Fix $k$ and consider $B:= \overline{\varphi_k^{(\mathcal{U})}(p_k)A_\mathcal{U} \varphi_k^{(\mathcal{U})} (p_k)}$, then we have 
	\begin{equation}
	\varphi^{(\mathcal{U})}(1_{\prod\limits_\mathcal{U} F^{(r)}})b = b \label{pseudo unit}
	\end{equation} 
	for all $b \in B$. Observe that for any free filter containing the cofinite filter, the last paragraph of the proof of \cite[Proposition 2.2]{dim-ncomparison} shows that the map $\varphi_k^{(\mathcal{U})}: \prod\limits_\mathcal{U} F^{(r)} \longrightarrow A_\mathcal{U}$ is order zero and, by the structure of order zero maps given in Theorem \ref{order zero}, we can write
	\begin{equation}
	\varphi_k^{(\mathcal{U})}(x)= \varphi_k^{(\mathcal{U})}(1_{\prod\limits_\mathcal{U} F^{(r)}}) \rho(x) = \rho(x) \varphi_k^{(\mathcal{U})}(1_{\prod\limits_\mathcal{U} F^{(r)}}),
	\end{equation}
	for a $^*$-homomorphism $\rho: \prod\limits_{\mathcal{U}} F^{(r)} \longrightarrow  \mathcal{M}\left(C^*\left( \varphi_k^{(\mathcal{U})}\left( \prod\limits_\mathcal{U} F^{(r)} \right) \right)\right) \huge\cap \left\{ \varphi_k^{(\mathcal{U})}(1_{\prod\limits_\mathcal{U} F^{(r)}}) \right\}'$. Thus
	\begin{eqnarray}
	\varphi_k^{(\mathcal{U})}(1_{\prod\limits_\mathcal{U} F^{(r)}})\varphi_k^{(\mathcal{U})}(p_k) &=& \varphi_k^{(\mathcal{U})}(	1_{\prod\limits_\mathcal{U} F^{(r)}})^2 \rho(p_k) = \rho(p_k)\varphi_k^{(\mathcal{U})}(1_{\prod\limits_\mathcal{U} F^{(r)}})^2 \\
	&=& \varphi_k^{(\mathcal{U})}(p_k) \varphi_k^{(\mathcal{U})}(1_{\prod\limits_\mathcal{U} F^{(r)}}).
	\end{eqnarray}
	Using this we obtain
	\begin{eqnarray}
	\varphi_k^{(\mathcal{U})}(1_{\prod\limits_\mathcal{U} F^{(r)}})\varphi_k^{(\mathcal{U})}(p_k) x \varphi_k^{(\mathcal{U})}(p_k) = \varphi_k^{(\mathcal{U})}(p_k)\varphi_k^{(\mathcal{U})}(1_{\prod\limits_\mathcal{U} F^{(r)}})x \varphi_k^{(\mathcal{U})}(p_k) \in B
	\end{eqnarray}
	for any $x \in A_\mathcal{U}$. Thus $\varphi_k^{(\mathcal{U})}(1_{\prod\limits_\mathcal{U} F^{(r)}}) b \in B$
	for all $b \in B$. Set
	\begin{equation}
	h = \frac{1}{1-\lambda_k}\sum_{j\neq k} \lambda_j \varphi_j^{(\mathcal{U})}(1_{\prod\limits_\mathcal{U} F^{(r)}}).
	\end{equation}
	By construction $h$ is a positive contraction and 
	\begin{equation}
	\varphi^{(\mathcal{U})}(1_{\prod\limits_\mathcal{U} F^{(r)}}) = \lambda_k \varphi_k^{(\mathcal{U})} (1_{\prod\limits_\mathcal{U} F^{(r)}}) + (1-\lambda_k)h.
	\end{equation}
	By Lemma \ref{multiplier lemma} and (\ref{pseudo unit}) we have
	\begin{equation}
	\varphi_k^{(\mathcal{U})}(1_{\prod\limits_\mathcal{U} F^{(r)}})b = b
	\end{equation}
	for all $b \in B$. By \cite[Proposition II.3.4.2 (ii)]{blackadar} $\varphi_k^{(\mathcal{U})}(p_k)$ is in $B$, so in particular we obtain
	\begin{equation}
	\varphi_k^{(\mathcal{U})}(p_k) = \varphi_k^{(\mathcal{U})}(1_{\prod\limits_\mathcal{U} F^{(r)}}) \varphi_k^{(\mathcal{U})}(p_k).
	\end{equation}
	Using the last identity and the fact that $\varphi_k^{(\mathcal{U})}$ is order zero, we obtain
	\begin{eqnarray}
	0 &=& \varphi_k^{(\mathcal{U})}(p_k)\varphi_k^{(\mathcal{U})}(1_{\prod\limits_\mathcal{U} F^{(r)}}-p_k) = \varphi_k^{(\mathcal{U})}(p_k)\varphi_k^{(\mathcal{U})}(1_{\prod\limits_\mathcal{U} F^{(r)}}) - \varphi_k^{(\mathcal{U})}(p_k)^2 \\
	& = & \varphi_k^{(\mathcal{U})}(p_k) - \varphi_k^{(\mathcal{U})}(p_k)^2
	\end{eqnarray}
	which means that $\varphi_k^{(\mathcal{U})}(p_k)$ is a projection as required.
\end{proof}

\begin{remark} \label{finite case of lemma}
	If there exists $m \in \mathbb{N}$ such that $\lambda_k = 0$ for $k > m$, then we can take $N=m$. 
\end{remark}

\begin{remark} \label{direct sum disjoint supports}
	Let $A$ be a $C^*$-algebra and let $\mathfrak{F}$ be a finite subset of $A$ and $\varepsilon >0$.
	Suppose there exists a CPC approximation $\left(F, \psi, \varphi \right)$ for $\mathfrak{F}$ within $\varepsilon$ with $\varphi=\sum\limits_{k=1}^{n} \lambda_k \varphi_k$ for some order zero maps $\varphi_k: F \longrightarrow A$ and coefficients $\lambda_k > 0$ such that $\sum\limits_{k=1}^{n} \lambda_k =1$. Set $F_k = F$ for $k=1, \cdots, n$ and define CPC maps $\widetilde{\psi}: A \longrightarrow \bigoplus\limits_{k=1}^n F_k$, $\widetilde{\varphi}:\bigoplus\limits_{k=1}^n F_k \longrightarrow A$ as $\widetilde{\psi}(a)=\psi(a)\oplus \cdots \oplus \psi(a)$ and $\widetilde{\varphi}(x_1 \oplus \cdots \oplus x_n) = \sum\limits_{k=1}^{n} \lambda_k \varphi_k\left(x_k\right)$. Since $\varphi \psi (a) = \widetilde{\varphi}\widetilde{\psi}(a)$ for all $a \in A$, $\left(\bigoplus\limits_{k=1}^n F_k, \widetilde{\psi}, \widetilde{\varphi} \right)$ is a CPC approximation for $\mathfrak{F}$ within $\varepsilon$; moreover, 
	for each $k$ the kernel of $\varphi_k$ contains $\bigoplus\limits_{i \neq k} F_i$.
\end{remark}

The following theorem is the main result of this work.

\begin{theorem}\label{main theo}
Let $A$ be a $C^*$-algebra. Suppose there exists $n \in \mathbb{N}$ such that for every finite subset $\mathfrak{F} \subset A$ and every $\varepsilon>0$ there exist CPC maps $\psi:A \longrightarrow F, \; \varphi: F \longrightarrow A$ where $F$ is a finite dimensional $C^*$-algebra and $\varphi$ is a convex combination of $n$ contractive order zero maps such that
\begin{equation}
\|a - \varphi \psi (a) \| < \varepsilon
\end{equation}
for all $a \in \mathfrak{F}$. Then $A$ is AF.
\end{theorem}

\begin{proof}
If $n=1$, the result follows from \cite[Theorem 3.4]{cov-dim}. Thus we can suppose $n \geq 2$. By the proof of \cite[Proposition 2.6]{nuclear-dimension}, any countable subset of $A$ is contained in a separable subalgebra satisfying the hypotheses of the theorem. Therefore, without loss of generality we may assume $A$ is separable. 

From the hypotheses, for any finite subset $\mathfrak{F}$ and any $\varepsilon>0$ there exist a CPC approximation $\left(F, \psi, \varphi \right)$ for $\mathfrak{F}$ within $\varepsilon$, order zero maps $\varphi_k : F \longrightarrow A$ and coefficients $\lambda^{(\mathfrak{F}, \varepsilon)}_k \geq 0$, for $k=1, \cdots, n$, such that $\sum\limits_{k=1}^{n} \lambda^{(\mathfrak{F}, \varepsilon)}_k=1$ and $\varphi= \sum\limits_{k=1}^{n} \lambda^{(\mathfrak{F},\varepsilon)}_k\varphi_k$.
By compactness of $[0,1]^n$, we may assume there are constants $\lambda_1, \cdots, \lambda_n \in [0,1]$ satisfying $\sum\limits_{k=1}^{n} \lambda_k = 1$ such that $\lambda^{(\mathfrak{F},\varepsilon)}_k = \lambda_k$ for any finite subset $\mathfrak{F}$ and $\varepsilon>0$.
Additionally we can suppose (renaming $n$ if necessary) that each $\lambda_k$ is strictly positive. Thus, by Remark \ref{direct sum disjoint supports}, for any $\mathfrak{F}$ and $\varepsilon>0$ there exists a CPC approximation $\left(\bigoplus\limits_{k=1}^n F_k , \psi, \varphi \right)$ for $\mathfrak{F}$ within $\varepsilon$ with $\varphi = \sum\limits_{k=1}^{n} \lambda_k \varphi_k$ where each $\varphi_k: F \longrightarrow A$ is an order zero map and $\bigoplus\limits_{i \neq k} F_i \subset \ker \varphi_k$.

By Lemma \ref{missing lemma}, there exist projections $p_k \in F_k$ for $1 \leq k \leq n$ such that:

\begin{enumerate}[(I)]
\item $\| \varphi(p \psi (a) p) - a \| < \varepsilon $ for all $a \in \mathfrak{F}$ with $p= \sum\limits_{k=1}^{n} p_k$,

\item $\| \varphi (p_k) - \varphi(p_k)\varphi(1_F) \| < \varepsilon$ where $1_F$ denotes the unit of $F$.
\end{enumerate}
Then we can produce, using a countable dense subset of $A$, a sequence of completely positive and contractive approximations
\[
\xymatrix{ A \ar[r]^{\psi^{(r)}} & F^{(r)} \ar[r]^{\varphi^{(r)}} & A
	}
\] 
satisfying the hypothesis of Lemma \ref{lemma 2 omega}. 
Observe that for any free filter $\mathcal{U}$ on $\mathbb{N}$ containing the cofinite filter, the last paragraph of the proof of \cite[Proposition 2.2]{dim-ncomparison} shows that the maps $\varphi_k^{(\mathcal{U})}: \prod\limits_\mathcal{U} F^{(r)} \longrightarrow A_\mathcal{U}$ are order zero.


We will apply Lemma \ref{lemma 2 omega} to replace the convex combination of order zero maps with convex combination of $^*$-homomorphisms. After this, we will proceed to replace the convex combination of $^*$-homomorphisms with exactly one of them. The choice of such a $^*$-homomorphism is not important as the estimates only depend on the corresponding coefficient by Lemma \ref{lemma 3 omega}. Therefore, in order to simplify the notation, we will choose the first one.

Fix $\mathfrak{F}$ and $\varepsilon > 0$ such that $\sqrt{\lambda_1^{-1}\varepsilon} < 1$. We can assume that any element in $\mathfrak{F}$ is positive of norm at most $1$. 
By Lemma \ref{lemma 2 omega} and Remark \ref{finite case of lemma}, there exists a completely positive and contractive approximation $\left(\bigoplus\limits_{k=1}^n F_k, \psi, \pi \right)$ such that 
\begin{equation}
\|a-\pi  \psi(a)\|< \varepsilon/3 \label{good approx 1}
\end{equation} 
for all $a \in \mathfrak{F}$ and $\pi = \sum_{k=1}^{n}\lambda_k \pi_k$ where each $\pi_k: \bigoplus\limits_{k=1}^n F_k \longrightarrow A$ is a $^*$-homomorphism satisfying $\bigoplus_{i \neq k} F_i \subset \ker \pi_k$.
 
Since the set of all minimal projections of $F_k$, $\mathcal{P}(F_k)$, is compact, we can find minimal projections $p_1,...,p_r \in \mathcal{P}(F)$ such that for all $p \in \mathcal{P}(F_k)$ and all $k$ there exists some $j \in \{1, \cdots, r\}$ such that
\begin{equation}
\|p - p_j\|< \frac{\lambda_1 \varepsilon^2}{3 \left(6M\right)^2} \label{good approx for proj}
\end{equation}
for some $j\in \{1,...,r\}$ where $M=\dim F$. Assume $p_j \in \mathcal{P}\left( F_{k_j} \right)$ and set
\begin{equation}
\mathfrak{F}'=\mathfrak{F} \cup \{\pi_{k_j}(p_j): 1\leq j \leq r \}.
\end{equation}
By Lemma \ref{lemma 2 omega} again, we find CPC maps $\psi':A \longrightarrow \bigoplus_{k=1}^n F'_k$ and $\theta:\bigoplus_{k=1}^n F'_k \longrightarrow A$ with $\theta = \sum_{k=1}^{n}\lambda_k \theta_k$, $F'_k$ finite dimensional $C^*$-algebras and each $\theta_k$ is a $^*$-homomorphism satisfying $\bigoplus\limits_{i\neq k} F'_i \subset \ker \theta_k$, 
such that 
\begin{equation}
\|a - \theta  \psi'(a) \| < \frac{\lambda_1\varepsilon^2}{3 \left(6M\right)^2} \label{good approx 2}
\end{equation} 
for all $a \in \mathfrak{F}'$. In particular for $p \in \mathcal{P}(F_k)$, let $p_j \in \mathfrak{F}'$ satisfy (\ref{good approx for proj}) so that  
\begin{eqnarray}
\|\pi_k(p) - \theta  \psi'(\pi_k(p)) \| &<& \| \pi_k(p) - \pi_k(p_j) \| + \| \pi_k(p_j) - \theta  \psi'(\pi_k(p_j)) \| \\
& & + \; \| \theta  \psi'(\pi_k(p_j)) - \theta  \psi'(\pi_k(p)) \|  \\ &<&
\frac{\lambda_1 \varepsilon^2}{ \left(6M\right)^2}.
\end{eqnarray}
Using that $\sqrt{\lambda_1^{-1}\varepsilon} < 1$ and Lemma \ref{lemma 3 omega}, we obtain 
\begin{equation}
\| \pi_k(p) - \theta_1  \psi' (\pi_k(p)) \| \leq \frac{\varepsilon}{3M} \label{3/M}
\end{equation}
for all $k$. For any $a \in \mathfrak{F}$, by the spectral theorem for Hermitian matrices, we can write 
\begin{equation}
\psi (a) = \sum_{i=1}^d t_i q_i
\end{equation}
with $0 \leq t_i \leq 1$ where $\{ q_i \in F: 1 \leq i \leq d \} $ is some set of minimal projections and $d \leq M$. Using the last identity and (\ref{3/M}) we have
\begin{eqnarray}
\left\| \pi \psi (a) - \theta_1 \psi' \pi \psi(a) \right\| &=& \left\| \sum_{i,k} t_i \lambda_k \pi_k(q_i) - \sum_{i,k} t_i \lambda_k \theta_1 \psi' \pi_k(q_i) \right\| \\
& \leq & \sum_{k=1}^{n} \lambda_k \left( \sum_{i=1}^{d} \| \pi_k(q_i) - \theta_1 \psi' \pi_k(q_i) \| \right) \\
& \leq & \sum_{k=1}^{n} \lambda_k \left( \frac{\varepsilon d}{3M} \right) \leq \frac{\varepsilon}{3}.
\end{eqnarray}

Finally, using the last inequality and (\ref{good approx 1}) we obtain
\begin{eqnarray}
\| a - \theta_1 \psi' (a) \| & \leq & \| a- \pi \psi(a) \| + \| \pi \psi (a) - \theta_1 \psi'\pi \psi(a)) \| \\
& & + \| \theta_1 \psi'(\pi \psi (a) - a ) \| \\
& < & \frac{\varepsilon}{3} + \frac{\varepsilon}{3} + \frac{\varepsilon}{3} = \varepsilon.
\end{eqnarray}
Thus $\dist(a,\theta_1(F'_1))< \varepsilon$ for all $a \in \mathfrak{F}$. Since $\theta_1: F'_1 \longrightarrow A$ is a $^*$-homomorphism and $F'_1$ is a finite dimensional $C^*$-algebra, $\theta_1 \left(F'_1\right)$ is also a finite dimensional algebra. Therefore $A$ is an AF-algebra.
\end{proof}

\begin{remark} \label{remark}
By the previous theorem, the decomposable approximations of a nuclear $C^*$-algebra $A$ given by \cite[Theorem 1.4]{Hir-Kir-Whi} can witness finite nuclear dimension (in fact, decomposition rank since $\varphi$ is forced to be contractive) if and only if $A$ is an approximately finite dimensional $C^*$-algebra.
\end{remark}

\subsection*{Acknowledgements} The author is grateful to his advisor, Stuart White, for his helpful insights and guidance throughout this work. The author would also like to thank Joan Bosa, Sam Evington, Gabriele Tornetta and Joachim Zacharias for their useful conversations during the early stages of this work. The author thanks the anonymous referee for their suggestions on an earlier version of this paper.

\bibliographystyle{plain}

\bibliography{papers}

\end{document}